\documentclass[12pt,english,a4paper]{smfart}

\usepackage[T1]{fontenc}
\usepackage{lmodern}
\usepackage{smfthm}
\usepackage[headings]{fullpage}
\usepackage{amssymb}

\newcommand{\Gal}{\operatorname{Gal}}
\newcommand{\Qp}{\mathbf{Q}_p}
\newcommand{\Qpbar}{\overline{\mathbf{Q}}_p}
\newcommand{\Qpunr}{\mathbf{Q}_p^\mathrm{unr}}
\newcommand{\Cp}{\mathbf{C}_p}

\newcommand{\ZZ}{\mathbf{Z}}

\newcommand{\OO}{\mathcal{O}}

\newcommand{\Card}{\operatorname{Card}}
\newcommand{\Aut}{\operatorname{Aut}}
\newcommand{\val}{\operatorname{val}}
\newcommand{\bigO}{\operatorname{O}}
\newcommand{\dcroc}[1]{[\![ #1 ]\!]}
\newcommand{\dpar}[1]{(\!( #1 )\!)}
\newcommand{\efont}{\mathbf{E}}
\newcommand{\afont}{\mathbf{A}}
\newcommand{\bfont}{\mathbf{B}}
\newcommand{\et}{\widetilde{\mathbf{E}}}
\newcommand{\at}{\tilde{\mathbf{A}}}
\newcommand{\bt}{\widetilde{\mathbf{B}}}
\newcommand{\btrig}[2]{\widetilde{\mathbf{B}}^{\dagger #1}_{\mathrm{rig} #2}}
\newcommand{\bnrig}[2]{\mathbf{B}^{\dagger #1}_{\mathrm{rig} #2}}
\newcommand{\LT}{\mathrm{LT}}
\renewcommand{\geq}{\geqslant}
\renewcommand{\leq}{\leqslant} 
\renewcommand{\phi}{\varphi} 

\author{Laurent Berger}
\address{UMPA de l'ENS de Lyon \\
UMR 5669 du CNRS}
\email{laurent.berger@ens-lyon.fr}
\urladdr{perso.ens-lyon.fr/laurent.berger/}

\title[The perfectoid commutant of Lubin-Tate power series]{The perfectoid commutant \\ of Lubin-Tate power series}

\date{\today}

\begin{document}

\begin{abstract}
Let $\LT$ be a Lubin-Tate formal group attached to a finite extension of $\Qp$. By a theorem of Lubin-Sarkis, an invertible characteristic $p$ power series that commutes with the elements $\Aut(\LT)$ is itself in $\Aut(\LT)$. We extend this result to perfectoid power series, by lifting such a power series to characteristic zero and using the theory of locally analytic vectors in certain rings of $p$-adic periods. This allows us to recover the field of norms of the Lubin-Tate extension from its completed perfection.
\end{abstract}



\maketitle

\setlength{\baselineskip}{18pt}
\section*{Introduction}

Let $F$ be a finite extension of $\Qp$, with ring of integers $\OO_F$ and residue field $k$. Let $q= \Card(k)$ and let $\pi$ be a uniformizer of $\OO_F$. Let $\LT$ be the Lubin-Tate formal $\OO_F$-module attached to $\pi$. Let $F_\infty = F(\LT[\pi^\infty])$ denote the extension of $F$ generated by the torsion points of $\LT$, and let $\Gamma_F = \Gal(F_\infty/F)$. The Lubin-Tate character $\chi_\pi$ gives rise to an isomorphism $\chi_\pi : \Gamma_F \to \OO_F^\times$.

The field of norms (\cite{W83}) $\efont_F$ of the extension $F_\infty/F$ is a local field of characteristic $p$, endowed with an action of $\Gamma_F$, that can be explicitly described as follows. We choose a coordinate $T$ on $\LT$, so that for each $a \in \OO_F$ we get a power series $[a](T) \in \OO_F \dcroc{T}$. We then have $\efont_F = k \dpar{ Y }$, on which $\Gamma_F$ acts via the formula $\gamma( f(Y)) = f( [\chi_\pi(\gamma)](Y))$. In $p$-adic Hodge theory, we consider the field $\et_F$, which is the $Y$-adic completion of the maximal purely inseparable extension $\cup_{n \geq 0} \efont_F^{q^{-n}}$ of $\efont_F$ inside an algebraic closure. The action of $\Gamma_F$ extends to the field $\et_F$. If $f \in \et_F$ and $\gamma \in \Gamma_F$, we still have  $\gamma( f(Y)) = f( [\chi_\pi(\gamma)](Y))$. The question that motivated this paper is the following.

\begin{enonce*}{Question} 
Can we recover $\efont_F$ from the data of the valued field $\et_F$ endowed with the action of $\Gamma_F$?
\end{enonce*}

If $a \in \OO_F^\times$, then $u(Y) = [a](Y)$ is an element of $\efont_F$ of valuation $1$ that satisfies the functional equation  $u \circ [g](Y)  = [g] \circ u(Y)$ for all $g \in \OO_F^\times$. Conversely, we prove the following theorem, which answers the question, as it allows us to find a uniformizer of $\efont_F$ from the data of the valued field $\et_F$ endowed with the action of $\Gamma_F$.

\begin{enonce*}{Theorem A}
If $u \in \et_F$ is such that $\val_Y(u)=1$ and $u \circ [g] = [g] \circ u$ for all $g \in \OO_F^\times$, then there exists $a \in \OO_F^\times$ such that $u(Y) = [a](Y)$.
\end{enonce*}

In particular, $\efont_F = k \dpar{u}$ for any $u$ as in theorem A. The main difficulty in the proof of theorem A is to prove that if $u$ is as in the statement of theorem A, then there exists $n \geq 0$ such that $u \in \efont_F^{q^{-n}}$. If $F=\Qp$ and $\pi=p$, namely in the cyclotomic situation, this follows from the main result of \cite{BR22}. However, a crucial ingredient in that paper does not generalize to $F \neq \Qp$. In order to go beyond the cyclotomic case, we instead use a result of Colmez (\cite{C02}) to lift $u$ to an element $\hat{u}$ of a ring $\at_F^+$ (the Witt vectors over the ring of integers of $\et_F$, as well as a completion of $\cup_{n \geq 0} \phi_q^{-n}(\OO_F \dcroc{\widehat{Y}})$, where $\phi_q(\widehat{Y}) = [\pi](\widehat{Y})$), that will satisfy a similar functional equation. In particular, $\hat{u}$ is a locally analytic element of a suitable ring of $p$-adic periods. By previous results of the author (\cite{B16}), $\hat{u}$ belongs to $\phi_q^{-n}(\OO_F \dcroc{\widehat{Y}})$ for a certain $n$. This allows us to prove that there exists $n \geq 0$ such that $u \in \efont_F^{q^{-n}}$. By replacing $u$ with $u^{p^k}$ for a well chosen $k$, we are led to the study of elements of $Y \cdot k \dcroc{Y}$ under composition. We prove that $u$ is invertible for composition, and to conclude we use a theorem of Lubin-Sarkis (\cite{LS07}) saying that if an invertible series commutes with a nontorsion element of $\Aut(\LT)$, then that series is itself in $\Aut(\LT)$. We finish this paper with an explanation of why the ``Tate traces'' on $\et_F$ used in \cite{BR22} don't exist if $F \neq \Qp$.

\section{Locally analytic vectors}

We use the notation that was introduced in the introduction. In order to apply lemma 9.3 of \cite{C02}, we assume that the coordinate $T$ on $\LT$ is chosen such that $[\pi](T)$ is a monic polynomial of degree $q$ (for example, we could ask that $[\pi](T)=T^q+\pi T$). 

Let $F_0 = \Qpunr \cap F$. Let $\et_F^+$ denote the ring of integers of $\et_F$ and let $\at^+_F = \OO_F \otimes_{\OO_{F_0}} W(\et_F^+)$ be the $\OO_F$-Witt vectors over $\et^+_F$.

\begin{prop}
\label{colift}
If $u \in \et_F^+$ is such that $\gamma(u)  = [\chi_\pi(\gamma)] (u)$ for all $\gamma \in \Gamma_F$, then $u$ has a lift $\hat{u} \in \at^+_F$ such that $\gamma(\hat{u})   = [\chi_\pi(\gamma)] \circ \hat{u}$ for all $\gamma \in \Gamma_F$.
\end{prop}

\begin{proof}
By lemma 9.3 of \cite{C02}, there is a unique lift $\hat{u} \in \at^+_F$ of $u$ such that $\phi_q(\hat{u}) = [\pi](\hat{u})$ (in ibid., this element is denoted by $\{u\}$). If $\gamma \in \Gamma_F$, then both $\gamma(\hat{u})$ and $[\chi_\pi(\gamma)](\hat{u})$ are lifts of $u$ that are compatible with Frobenius as above. By unicity, they are equal.
\end{proof}

Let $\log_{\LT}(T)$ and $\exp_{\LT}(T)$ be the logarithm and exponential series for $\LT$. Write $\exp_{\LT}(T) = \sum_{n \geq 1} e_n T^n$ and $\exp_{\LT}(T)^j = \sum_{n \geq j} e_{j,n} T^n$ for $j \geq 1$.

\begin{lemm}
\label{valexp}
We have $\val_\pi(e_{j,n}) \geq -n/(q-1)$ for all $j,n \geq 1$.
\end{lemm}

\begin{proof}
Fix $\varpi \in \Qpbar$ such that $\val_\pi(\varpi) = 1/(q-1)$ and let $K=F(\varpi)$. Recall that $\log_{\LT}(T) = \lim_{n \to  +\infty} [\pi^n](T)/\pi^n$. If $z \in \Cp$ and $\val_\pi(z) \geq 1/(q-1)$, then $\val_\pi([\pi](z)) \geq \val_\pi(z)+1$. 
This implies that $1/\varpi \cdot \log_{\LT}(\varpi T) \in T + T^2 \OO_K \dcroc{T}$. Its composition inverse $1/\varpi \cdot \exp_{\LT}(\varpi T)$ therefore also belongs to $T + T^2 \OO_K \dcroc{T}$. This implies the claim for $j=1$. The claim for $j \geq 1$ follows easily.
\end{proof}

We use a number of rings of $p$-adic periods in the Lubin-Tate setting, whose construction and properties were recalled in \S 3 of \cite{B16}. Proposition \ref{colift} gives us an element $\hat{Y} \in \at_F^+$ (denoted by $u$ in ibid.). Let $\bt_F^+=\at_F^+[1/\pi]$. Given an interval $I = [r;s] \subset [0;+\infty[$, a valuation $V(\cdot,I)$ on $\bt^+_F[1/\hat{Y}]$ is constructed in ibid., as well as various completions of that ring. We use $\bt_F^I$, the completion of $\bt^+_F[1/\hat{Y}]$ for $V(\cdot,I)$ and $\btrig{,r}{,F} = \varprojlim_{s \geq r} \bt_F^{[r;s]}$. Inside $\btrig{,r}{,F}$, there is the ring $\bnrig{,r}{,F}$ of power series $f(\hat{Y})$ with coefficients in $F$, where $f(T)$ converges on a certain annulus depending on $r$.

\begin{lemm}
\label{atpbrdag}
If $s \geq 0$, then $\bnrig{,s}{,F} \cap \at^+_F = \afont^+_F$.
\end{lemm}

\begin{proof}
Take $f(\hat{Y}) \in \bnrig{,s}{,F}$, $t \geq s$ and let $I=[s;t]$. We have $V(f,I)\geq 0$, so that $f$ is bounded by $1$ on the corresponding annulus. This is true for all $t$, so that $f \in \bfont^{\dagger,s}_F$. We now have $f \in \bfont^{\dagger,s}_F \cap \at^+_F = \afont_F^+$.
\end{proof}

Let $W$ be a Banach space with a continuous action of $\Gamma_F$. The notion of locally analytic vector was introduced in \cite{ST03}. Recall (see for instance \S 2 of \cite{B16}; the definition given there is easily seen to be equivalent to the following one) that an element $w \in W$ is locally $F$-analytic if there exists a sequence $\{w_k\}_{k \geq 0}$ of $W$ such that $w_k \to 0$, and an integer $n \geq 1$ such that for all $\gamma \in \Gamma_F$ such that $\chi_\pi(\gamma)=1+p^n c(\gamma)$ with $c(\gamma) \in \OO_F$, we have $\gamma(w) = \sum_{k \geq 0} c(\gamma)^k w_k$. If $W = \varprojlim_ i W_i$ is a Fr\'echet representation of $\Gamma_F$, we say that $w \in W$ is pro-$F$-analytic if its image in $W_i$ is locally $F$-analytic for all $i$.

\begin{prop}
\label{islocan}
If $r \geq 0$ and $x \in \at^+_F$ is such that $\val_Y(\overline{x})>0$ and $\gamma(x) = [\chi_\pi(\gamma)] (x)$ for all $\gamma \in \Gamma_F$, then $x$ is a pro-$F$-analytic element of $\btrig{,r}{,F}$.
\end{prop}

\begin{proof}
We prove that for all $s \geq r$, $x$ is a locally $F$-analytic vector of $\bt_F^{[r;s]}$. The proposition then follows, since $\btrig{,r}{,F} = \varprojlim_{s \geq r} \bt_F^{[r;s]}$ as Fr\'echet spaces. 

Let $S(X,Y) = \sum_{i,j} s_{i,j} X^i Y^j \in \OO_F \dcroc{X,Y}$ be the power series that gives the addition in $\LT$. We have $\log_{\LT}(x) \in \bt_F^{[r;s]}$. Take $n \geq 1$ such that $V(p^{n-1} \log_{\LT}(x),[r;s]) > 0$. We have $[a](T) = \exp_{\LT}(a \log_{\LT}(T))$, so that $[1+p^n c](T) = S(T,\exp_{\LT}(p^n c \log_{\LT}(T)))$. If $\chi_\pi(\gamma) = 1+p^n c(\gamma)$, then 
\begin{align*} 
\gamma(x) & = \sum_{k \geq 0} c(\gamma)^k \sum_{j \leq k} p^{nk} e_{j,k} \log_{\LT}(x)^k \sum_{i \geq 0} s_{i,j} x^i  \\ & = \sum_{k \geq 0} c(\gamma)^k \sum_{j \leq k} p^k e_{j,k} \cdot (p^{n-1} \log_{\LT}(x))^k \cdot \sum_{i \geq 0} s_{i,j} x^i . 
\end{align*}
We have $p^k e_{j,k} \in \OO_F$ by lemma \ref{valexp}, $V(p^{n-1} \log_{\LT}(x),[r;s]) > 0$ by hypothesis, $s_{i,j} \in \OO_F$ and $V(x,[r;s])>0$. This implies the claim.
\end{proof}

\begin{prop}
\label{locanpr}
If $r > 0$ and $x \in \at^+_F$ is a pro-$F$-analytic element of $\btrig{,r}{,F}$, then there exists $n \geq 0$ such that $x \in \phi_q^{-n}(\afont^+_F)$.
\end{prop}

\begin{proof}
By item (3) of theorem 4.4 of \cite{B16} (applied with $K=F$), there exists $n \geq 0$ and $s>0$ such that $x \in \phi_q^{-n} (\bnrig{,s}{,F})$. The proposition now follows from lemma \ref{atpbrdag} applied to $\phi_q^n(x)$.
\end{proof}

\section{Composition of power series}

Recall that a power series $f(Y) \in k\dcroc{Y}$ is separable if $f'(Y) \neq 0$. If $f(Y) \in Y \cdot k \dcroc{Y}$, we say that $f$ is invertible if $f'(0) \in k^\times$, which is equivalent to $f$ being invertible for composition (denoted by $\circ$). We say that $w(Y) \in Y \cdot k \dcroc{Y}$ is nontorsion if $w^{\circ n}(Y) \neq Y$ for all $n \geq 1$. If $w(Y) = \sum_{i \geq 0} w_i Y^i \in k \dcroc{Y}$ and $m \in \ZZ$, let $w^{(m)}(Y) = \sum_{i \geq 0} w_i^{p^m} Y^i$. Note that $(w \circ v)^{(m)} = w^{(m)} \circ v^{(m)}$.

\begin{prop}
\label{lubnarch}
Let $w(Y) \in Y + Y^2\cdot k \dcroc{Y}$ be an invertible nontorsion series, and let $f(Y) \in Y \cdot k \dcroc{Y}$ be a separable power series. If $w^{(m)} \circ f = f \circ w$, then $f$ is invertible.
\end{prop}

\begin{proof}
This is a slight generalization of lemma 6.2 of \cite{L94}. Write 
\begin{align*}
f(Y) & = f_n Y^n + \bigO(Y^{n+1}) \\ f'(Y)  &= g_k Y^k + \bigO(Y^{k+1}) \\ w(Y) & = Y + w_r Y^r + \bigO(Y^{r+1}),
\end{align*}
with $f_n,g_k,w_r \neq 0$. Since $w$ is nontorsion, we can replace $w$ by $w^{\circ p^\ell}$ for $\ell \gg 0$ and assume that $r \geq  k+1$. We have 
\begin{align*} w^{(m)} \circ f & = f(Y) + w_r^{(m)} f(Y)^r + \bigO(Y^{n(r+1)})  \\ & = f(Y) + w_r^{(m)} f_n^r Y^{nr} + \bigO(Y^{nr+1}). \end{align*}
If $k=0$, then $n=1$ and we are done, so assume that $k \geq 1$. We have
\begin{align*} f \circ w &= f(Y + w_r Y^r + \bigO(Y^{r+1})) \\ &= f(Y) + w_r Y^r f'(Y)  + \bigO(Y^{2r})  \\ &=  f(Y) + w_r g_k Y^{r+k} + \bigO(Y^{r+k+1}). \end{align*}
This implies that $nr=r+k$, hence $(n-1)r=k$, which is impossible if $r>k$ unless $n=1$. Hence $n=1$ and $f$ is invertible.
\end{proof}

We now prove theorem A. Take $u \in \et_F$ such that $\val_Y(u)=1$ and $u \circ [g] = [g] \circ u$ for all $g \in \OO_F^\times$. By proposition \ref{colift}, $u$ has a lift $\hat{u} \in \at^+_F$ such that $\gamma(\hat{u})   = [\chi_\pi(\gamma)] \circ \hat{u}$ for all $\gamma \in \Gamma_F$. By proposition \ref{islocan}, $\hat{u}$ is a pro-$F$-analytic element of $\btrig{,r}{,F}$. By proposition \ref{locanpr}, there exists $n \geq 0$ such that $\hat{u} \in \phi_q^{-n}(\afont^+_F)$. This implies that $u \in \phi_q^{-n}(\efont^+_F)$. Hence there is an $m \in \ZZ$ such that $f(Y) = u(Y)^{p^m}$ belongs to $Y \cdot k \dcroc{Y}$ and is separable. Note that $\val_Y(f) = p^m$. Take $g \in 1+\pi \OO_F$ such that $g$ is nontorsion, and let $w(Y) = [g](Y)$ so that $u \circ w = w \circ u$. We have $f \circ w = w^{(m)} \circ f$. By proposition \ref{lubnarch}, $f$ is invertible. This implies that $\val_Y(f)=1$, so that $m=0$ and $u$ itself is invertible. Since $u \circ [g] = [g] \circ u$ for all $g \in \OO_F^\times$, theorem 6 of \cite{LS07} implies that $u \in \Aut(\LT)$. Hence there exists $a \in \OO_F^\times$ such that $u(Y)=[a](Y)$.

\section{Tate traces in the Lubin-Tate setting}

If $F=\Qp$ and $\pi=p$ (namely in the cyclotomic situation) the fact that, in the proof of theorem A, there exists $n \geq 0$ such that $u \in \phi_q^{-n}(\efont^+_F)$ follows from the main result of \cite{BR22}. We now  explain why the methods of ibid don't extend to the Lubin-Tate case. More precisely, we prove that there is no $\Gamma_F$-equivariant $k$-linear projector $\et_F \to \efont_F$ if $F \neq \Qp$. Choose a coordinate $T$ on $\LT$ such that $\log_{\LT}(T) = \sum_{n \geq 0} T^{q^n}/\pi^n$, so that $\log'_{\LT}(T) \equiv 1 \bmod{\pi}$. Let $\partial = 1/\log'_{\LT}(T) \cdot d/dT$ be the invariant derivative on $\LT$.

\begin{lemm}
\label{ltder}
We have $d \gamma(Y)/dY \equiv \chi_\pi(\gamma)$ in $\efont_F$ for all $\gamma \in \Gamma_F$.
\end{lemm}

\begin{proof}
Since $\log'_{\LT} \equiv 1 \bmod{\pi}$, we have $\partial = d/dY$ in $\efont_F$. Applying $\partial \circ \gamma =  \chi_\pi(\gamma) \gamma \circ \partial$ to $Y$, we get the claim.
\end{proof}

\begin{lemm}
\label{pginv}
If $\gamma \in \Gamma_F$ is nontorsion, then $\efont_F^{\gamma=1} = k$.
\end{lemm}

\begin{prop}
\label{noltrace}
If $F \neq \Qp$, there is no $\Gamma_F$-equivariant map $R : \efont_F \to \efont_F$ such that $R(\phi_q(f)) = f$ for all $f \in \efont_F$.
\end{prop}

\begin{proof}
Suppose that such a map exists, and take $\gamma \in \Gamma_F$ nontorsion and such that $\chi_\pi(\gamma) \equiv 1 \bmod{\pi}$. We first show that if $f \in \efont_F$ is such that $(1- \gamma) f \in \phi_q(\efont_F)$, then $f \in \phi_q(\efont_F)$. Write $f=f_0 + \phi_q( R(f))$ where $f_0 = f- \phi_q(R(f))$, so that $R(f_0)=0$ and $(1- \gamma) f_0  = \phi_q(g) \in \phi_q(\efont_F)$. Applying $R$, we get $0 = (1- \gamma) R(f_0) = g$. Hence $g=0$ so that $(1- \gamma) f_0 = 0$. Since $\efont_F^{\gamma=1} = k$ by lemma \ref{pginv}, this implies $f_0\in k$, so that $f \in \phi_q(\efont_F)$.

However, lemma \ref{ltder} and the fact that $\chi_\pi(\gamma) \equiv 1 \bmod{\pi}$ imply that $\gamma(Y) = Y + f_\gamma(Y^p)$ for some $f_\gamma \in \efont_F$, so that $\gamma(Y^{q/p}) = Y^{q/p} + \phi_q(g_\gamma)$. Hence $(1-\gamma)(Y^{q/p}) \in \phi_q(\efont_F)$ even though $Y^{q/p}$ does not belong to $\phi_q(\efont_F)$. Therefore, no such map $R$ can exist.
\end{proof}

\begin{coro}
\label{noproj}
If $F \neq \Qp$, there is no $\Gamma_F$-equivariant $k$-linear projector $\phi_q^{-1}(\efont_F) \to \efont_F$. 
A fortiori, there is no $\Gamma_F$-equivariant $k$-linear projector $\et_F \to \efont_F$.
\end{coro}

\begin{proof}
Given such a projector $T$, we could define $R$ as in prop \ref{noltrace} by $R = T \circ \phi_q^{-1}$.
\end{proof}

\vspace{\baselineskip}\noindent\textbf{Acknowledgements.}
I thank Juan Esteban Rodr\'{\i}guez Camargo for asking me the question that motivated both this paper and \cite{BR22}.

\providecommand{\bysame}{\leavevmode ---\ }
\providecommand{\og}{``}
\providecommand{\fg}{''}
\providecommand{\smfandname}{\&}
\providecommand{\smfedsname}{\'eds.}
\providecommand{\smfedname}{\'ed.}
\providecommand{\smfmastersthesisname}{M\'emoire}
\providecommand{\smfphdthesisname}{Th\`ese}


\begin{thebibliography}{Win83}

\bibitem[Ber16]{B16}
{\scshape L.~Berger} -- {\og Multivariable {$(\varphi,\Gamma)$}-modules and
  locally analytic vectors\fg}, \emph{Duke Math. J.} \textbf{165} (2016),
  no.~18, p.~3567--3595.

\bibitem[BR22]{BR22}
{\scshape L.~Berger {\normalfont \smfandname} S.~Rozensztajn} -- {\og
  Decompletion of cyclotomic perfectoid fields in positive characteristic\fg},
  preprint, 2022.

\bibitem[Col02]{C02}
{\scshape P.~Colmez} -- {\og Espaces de {B}anach de dimension finie\fg},
  \emph{J. Inst. Math. Jussieu} \textbf{1} (2002), no.~3, p.~331--439.

\bibitem[LS07]{LS07}
{\scshape J.~Lubin {\normalfont \smfandname} G.~Sarkis} -- {\og Extrinsic
  properties of automorphism groups of formal groups\fg}, \emph{J. Algebra}
  \textbf{315} (2007), no.~2, p.~874--884.

\bibitem[Lub94]{L94}
{\scshape J.~Lubin} -- {\og Nonarchimedean dynamical systems\fg},
  \emph{Compositio Math.} \textbf{94} (1994), no.~3, p.~321--346.

\bibitem[ST03]{ST03}
{\scshape P.~Schneider {\normalfont \smfandname} J.~Teitelbaum} -- {\og
  Algebras of {$p$}-adic distributions and admissible representations\fg},
  \emph{Invent. Math.} \textbf{153} (2003), no.~1, p.~145--196.

\bibitem[Win83]{W83}
{\scshape J.-P. Wintenberger} -- {\og Le corps des normes de certaines
  extensions infinies de corps locaux; applications\fg}, \emph{Ann. Sci.
  \'{E}cole Norm. Sup. (4)} \textbf{16} (1983), no.~1, p.~59--89.

\end{thebibliography}
\end{document}